\documentclass[11pt,reqno]{amsart}
\usepackage[utf8]{inputenc}
\usepackage[english]{babel}
\usepackage[babel]{csquotes}
\usepackage[backend=bibtex, hyperref]{biblatex}
\bibliography{RMT}

\usepackage{amsmath}
\usepackage{amsfonts}
\usepackage{amssymb}
\usepackage{amsthm} 
\usepackage{dsfont} 
\usepackage[mathscr]{euscript} 
\usepackage{eurosym}
\usepackage{bbm} 
\usepackage[a4paper]{geometry}
\usepackage{comment} 
\usepackage{paralist} 

\usepackage{graphicx}
\usepackage{xkeyval}
\usepackage{pstricks}
\usepackage{hyperref}
\usepackage{longtable}
\usepackage{tikz}

\hypersetup{
urlcolor=black, 
  menucolor=black, 
  citecolor=black, 
  anchorcolor=black, 
  filecolor=black, 
  linkcolor=black, 
  colorlinks=true,
}
\usepackage{comment}

\usetikzlibrary{matrix}

\newcommand{\lebesgue}{\lambda}

\newcommand{\Prob}{\mathds{P}}

\newcommand{\E}{\mathds{E}}

\newcommand{\C}{\mathbb{C}}
\newcommand{\R}{\mathbb{R}}

\newcommand{\N}{\mathbb{N}}

\newcommand{\abs}[1]{|{#1}|} 
\newcommand{\bigabs}[1]{\left|{#1}\right|} 
\newcommand{\one}{\mathds{1}}

\newcommand{\Acal}{\mathcal{A}}

\newcommand{\Bcal}{\mathcal{B}}

\newcommand{\Lcal}{\mathcal{L}}
\newcommand{\Zcal}{\mathcal{Z}}

\newcommand*{\defeq}{\mathrel{\vcenter{\baselineskip0.5ex \lineskiplimit0pt
                     \hbox{\scriptsize.}\hbox{\scriptsize.}}}%
                     =}

\newcommand{\de}{\text{d}}

\newcommand{\norm}[1]{\|#1\|}
\newcommand{\bignorm}[1]{\left\|#1\right\|}

\newcommand{\oneto}[1]{[{#1}]}

\renewcommand{\Im}{\operatorname{Im}}

\DeclareMathOperator{\tr}{tr}
\DeclareMathOperator{\diag}{diag}

\theoremstyle{plain}
\newtheorem{lemma}{Lemma}
\newtheorem{theorem}[lemma]{Theorem}

\newtheorem{corollary}[lemma]{Corollary}
\theoremstyle{definition}
\newtheorem{definition}[lemma]{Definition}

\newtheorem{example}[lemma]{Example}
\newtheorem{remark}[lemma]{Remark}
\theoremstyle{remark}

\newcommand{\1}{\one}
\definecolor{darkblue}{rgb}{.1, 0.1,.8}
\definecolor{darkgreen}{rgb}{0,0.8,0.2}
\definecolor{darkred}{rgb}{.8, .1,.1}

\DeclareMathOperator{\e}{e}

\newcommand{\nto}{n \to \infty}
\newcommand{\pto}{p \to \infty}

\newcommand{\var}{{\rm var}}

\newcommand{\as}{{\rm a.s.}}

\newcommand{\MP}{Mar\v cenko--Pastur }

\newcommand{\cip}{\stackrel{\Prob}{\rightarrow}}

\begin{document}
\title[Curie--Weiss Matrices]
{High-dimensional sample covariance matrices with Curie-Weiss entries}
\thanks{M.F.'s research was supported by the FernUniversität in Hagen. J.H.~was supported by the Deutsche Forschungsgemeinschaft (DFG) via RTG 2131 High-dimensional Phenomena in Probability – Fluctuations and Discontinuity.}
\author[Michael Fleermann]{Michael Fleermann}
\author[Johannes Heiny]{Johannes Heiny}
\begin{abstract}
We study the limiting spectral distribution of sample covariance matrices $XX^T$, where $X$ are $p\times n$ random matrices with correlated entries, for the cases $p/n\to y\in [0,\infty)$. If $y>0$, we obtain the \MP distribution and in the case $y=0$ the semicircle distribution (after appropriate rescaling). The entries we consider are Curie-Weiss spins, which are correlated random signs, where the degree of the correlation is governed by an inverse temperature $\beta>0$. The model exhibits a phase transition at $\beta=1$. The correlation between any two entries decays at a rate of $O(np)$ for $\beta \in (0,1)$, $O(\sqrt{np}$) for $\beta=1$, and for $\beta>1$ the correlation does not vanish in the limit. In our proofs we use Stieltjes transforms and concentration of random quadratic forms.
\end{abstract}
\keywords{Curie--Weiss, random matrix, \MP law, semicircle law, high dimension, dependent entries, full correlation}
\subjclass{\textit{2010 MSC:} Primary 60B20; Secondary 60F05 60F10 60G10 60G55 60G70} 
\maketitle

\section{Introduction and Preliminaries}

 In many contemporary applications, one is 
faced with large data sets where both the dimension of the observations and the sample size are large.  In quantum mechanics, for example, the energy levels of particles in a large system can
be approximated by the eigenvalues of a large random matrix. Estimating the underlying covariance structure of high-dimensional data with the sample covariance matrix can be misleading \cite{bai:silverstein:2010,elkaroui:2009}. Even in the case of independent covariates, it is well-known that the sample covariance matrix poorly estimates the population covariance matrix. The fluctuations of the off-diagonal entries of the sample covariance matrix aggregate, creating an estimation bias which was quantified in 1967 by the famous \MP theorem  \cite{marchenko:pastur:1967}. Ever since, the classical setting of well-behaved i.i.d.\ ensembles was extended to investigate settings more aligned with reality. In many situations, it is reasonable to assume that entries in data sets are dependent. The dependence might span between different observations, but also between covariates of individual observations. In random matrix theory, one often considers models exhibiting linear dependence between the entries. Works that consider non-linear dependencies are sparse. The paper \cite{bai:zhou:2008}, for example, incorporates non-linear dependence within the columns of the data matrix, but assumes these columns to be independent. In this paper, we consider a data matrix filled with Curie-Weiss spins. This model exhibits nonlinear dependence between all entries.  For technical reasons, settings with correlated entries are harder to analyze, since many proof techniques break down in presence of correlations.

Another way to deviate from the classical setting is to assume that data might stem from heavy-tailed distributions. The theory for the eigenvalues and eigenvectors of the sample covariance matrices stemming from heavy-tailed time series with infinite fourth moment is quite different from the classical \MP theory which applies in the light-tailed case.
For detailed discussions about classical random matrix theory, we refer to the monographs \cite{bai:silverstein:2010,yao:zheng:bai:2015}, 
while the developments in the heavy-tailed case can be found in \cite{davis:mikosch:pfaffel:2016,davis:heiny:mikosch:xie:2016,heiny:mikosch:2017:iid,auffinger:arous:peche:2009, heiny:mikosch:2017:corr,heiny:mikosch:2019} and the references therein.

The \MP law gives insight into the spectrum of large dimensional sample covariance matrices. Assume we have $n$ observations $x_1,\ldots,x_n$, each with $p$ real-valued covariates, where $n, p\in\N$, so that $x_i=(x_i(1),\ldots,x_i(p))^T$ for all $i\in\{1,\ldots,n\}$. Define the $p\times n$ data matrix $X_n \defeq (x_1,x_2,\ldots,x_n)$, that is, $X_n$ has columns $x_i$. The (centered) sample covariance matrix is then defined by
\[
\tilde{V}_n \defeq \frac{1}{n-1} \sum_{k=1}^n (x_k-\bar{x})(x_k-\bar{x})^T,
\]
which is of dimension $p\times p$. Here, the vector $\bar{x}$ denotes the arithmetic mean of the vectors $x_k$. Assuming that the data stems from $n$ i.i.d.\ realizations of an $\R^p$-valued random vector $x$ with $\Lcal_2$-entries, $\tilde{V}_n$ is an unbiased estimator for its covariance matrix $\var(x)$.

The sample covariance matrix is of crucial importance in multivariate
statistics, for instance in principal component analysis, canonical correlation analysis,
multivariate regression, factor analysis, hypothesis testing and discriminant analysis.
Many test statistics are based on the eigenvalues of the sample covariance matrix. Examples include independence tests \cite{bodnar:dette:parolya:2019} and likelihood ratio tests. For the latter it is essential that the log-determinant of $\tilde{V}_n$ can be written as $\log(\lambda_1)+ \cdots + \log(\lambda_p)$, where $(\lambda_i)$ are the eigenvalues of $\tilde{V}_n$.

When analyzing the limiting spectral distribution (LSD) of the eigenvalues, it suffices to consider the (non-centered) sample covariance matrix
\begin{equation}
\label{eq:Sn}	
V_n \defeq \frac{1}{n}\sum_{k=1}^n x_k x_k^T = \frac{1}{n}X_nX_n^T,
\end{equation}
since $\bar{x}\bar{x}^T$ is of rank $1$, see Theorem A.44 in \cite{bai:silverstein:2010}. From now on we will refer to $V_n$ as the sample covariance matrix. Our object of interest in this paper will be the limit of the empirical spectral distributions (ESD) $F_{V_n}$ defined as
\begin{equation*}
F_{V_n}(x)= \frac{1}{p}\; \sum_{i=1}^p \1_{\{ \lambda_i(V_n)\le x \}}, \qquad x\in  \R\,,
\end{equation*} 
where $\lambda_1(V_n)\ge \cdots \ge \lambda_p(V_n)$ are the ordered eigenvalues of $V_n$. If such a limit exists in the sense of weak convergence almost surely, we call it the limiting spectral distribution of $V_n$.

Also, we will assume that the number of covariates $p$ and the sample size $n$ are large and tend to infinity together.
In this paper, the sample size $n$ is a function of the dimension $p$ (cf.\ Remark~{rem:npdependence}) and the dimension increases at most proportionally to the sample size. To be precise, we assume
\begin{equation}\label{Cgamma}
n=n_p \to \infty \quad \text{ and } \quad \frac{p}{n_p}\to y\in [0,\infty)\,,\quad \text{ as } \pto\,.
\end{equation}
The constant $y$ controls the growth of the dimension relative to the sample size. Most of the random matrix literature focuses exclusively on the case $y>0$, while the case $y=0$ plays only a minor role. In many fields, however, the wider range of possible growth rates arising in the $y=0$ regime is desirable. The framework in this paper unifies these two lines of research.

\subsection{Background}

Before we present our model, we provide some background.
Assume that the entries of $X_n$ are i.i.d.\ with unit variance and zero mean. Then if $p/n \to y\in (0,\infty)$ the limiting spectral distribution of $(V_n)$ is the so-called \MP distribution $\mu^y$.
The (standard) MP distribution with ratio index $y\in(0,\infty)$ is the probability measure $\mu^{y}$ on $(\R,\Bcal)$ given by
\[
\mu^{y} = \frac{1}{2\pi xy}\sqrt{((1+\sqrt{y})^2-x)(x-(1-\sqrt{y})^2)} \one_{((1-\sqrt{y})^2,(1+\sqrt{y})^2)}(x)\lebesgue(\de x) + \left(1-\frac{1}{y}\right)\delta_0 \one_{y>1},
\]
where $\lebesgue$ denotes the Lebesgue measure on $(\R,\Bcal)$ and $\delta_0$ denotes the Dirac measure in $0$.

It is well-known that measures on $\R$ are uniquely characterized by their Stieltjes transforms \cite{yao:zheng:bai:2015}. The Stieltjes transform of $\mu^{y}$ is given for $z\in\C_+=\{c\in \C: \Im(c)>0\}$ by
\begin{equation*}\label{eq:MPstieltjes}
S_{\mu^{y}}(z):=\int_{\R} \frac{1}{x-z} \mu^y(dx) 
=\frac{1-y-z+\sqrt{(1-y-z)^2-4yz}}{2yz},
\end{equation*}
where throughout this paper, if $ z\in\R_+$, $\sqrt{z}$ denotes the positive square root, while if $z\in\C\backslash\R_+$, then $\sqrt{z}$ denotes the complex square root with positive imaginary part; see for example \cite{bai:silverstein:2010}. 
If $p/n \to \infty$, we observe $\delta_0$ as LSD of $V_n$, as there are at most $\min(p,n)$ positive eigenvalues of $V_n$. 

In the case $p/n \to 0$, the limiting spectral distribution of $V_n$ is the Dirac measure at $1$. After centering $V_n$ by the identity matrix $I$ and a subsequent appropriate rescaling, one can obtain a non-degenerate limiting spectral distribution. In \cite{bai:yin:1988} it is proved under the additional assumption $\E[X_n(1,1)^4]<\infty$ that the empirical spectral distribution of the matrices $\sqrt{n/p}\, (V_n- I)$
converges to the semicircle law $G$ with Lebesgue density
\begin{equation*}\label{eq:semicircle}
g(x)= \tfrac{1}{2\pi} \sqrt{4-x^2} \1_{[-2,2]}(x) \,,\qquad x\in\R,
\end{equation*}
and Stieltjes transform
\begin{equation*}
s_G(z)=\frac{-z + \sqrt{z^2-4}}{2}\,,\qquad z\in\C_+.
\end{equation*}
The i.i.d.\ assumption on the entries of the data matrix $X_n$ can be relaxed to linear dependence of the form $\Sigma_n^{1/2}X_n$ for symmetric positive definite deterministic matrices $\Sigma_n$ with uniformly bounded spectral or operator norm $\|\Sigma_n\|:= \sqrt{\lambda_1(\Sigma_n \Sigma_n^T)}$. For $p/n\to y>0$, the Stieltjes transform of the LSD of $n^{-1} \Sigma_n^{1/2}X_n X_n^T \Sigma_n^{1/2}$ can then be characterized via the LSD of $\E[V_n]=\Sigma_n$; see \cite{bai:silverstein:2010} for details. The same holds in the case $p/n\to 0$ for the LSD of $\sqrt{n/p}\, (n^{-1} \Sigma_n^{1/2}X_n X_n^T \Sigma_n^{1/2}- \Sigma_n)$ as proved in \cite{pan2012asymptotic} and \cite{wang:paul:2014}.

It is important to note that the linear dependence between the entries of $X_n$ was a crucial assumption for the above results. For nonlinear dependencies the situation becomes more delicate as the following examples will show. We present two examples of random matrices $Y_n$ with dependent entries and $\E[n^{-1} Y_n Y_n^T]=I$ for which the LSD of $n^{-1} Y_n Y_n^T$ is not the \MP distribution $\mu^y$. 

\begin{example} Assume that the entries of $X_n$ are i.i.d.\ continuous random variables with unit variance, zero mean and let $p/n\to y>0$.
Kendall's Tau is a U-statistic which measure the association of random variables. For higher dimensional observations, such as the columns $x_i=(x_i(1),\ldots,x_i(p))^T$
of the data matrix $X$,  the (empirical) Kendall's Tau matrix is defined as 
\begin{equation*}
\tau_n=\tfrac{2}{n(n-1)}  \sum_{1\le s<t\le n} \operatorname{sign}(x_s-x_t) (\operatorname{sign}(x_s-x_t))^T\,, 
\end{equation*}
where $\operatorname{sign}$ of a vector is taken coordinatewise. In particular, one sees that $\tau_n(i,i)=1$. Since $X_n$ has i.i.d.\ continuous entries, we have $\E[\tau_n]=I$.
Bandeira et al.\ \cite{bandeira:lodhia:rigollet:2017} proved that the empirical spectral distribution $F_{\tau_n}$ of $\tau_n$ converges, namely
\begin{equation*}\label{eq:rigollet}
F_{\tau_n}\cip \tfrac{2}{3} \, \xi + \tfrac{1}{3}\,, \quad p \to \infty\,,
\end{equation*}
where the random variable $\xi$ has a \MP distribution with parameter $y$. In Theorem \ref{thm:CWMP} we will observe a similar scaling phenomenon. The exact formula for $Y_n$ such that $n^{-1} Y_n Y_n^T=\tau_n$ can be found in \cite{bao:2017}.
\end{example}
\begin{example}
Assume that the entries of $X_n$ are i.i.d.\ symmetric random variables with tails $\Prob(|x_1(1)|>t)=t^{-\alpha} \ell(t)$, where $\alpha\in (0,2)$ and $\ell$ is a slowly varying function at infinity. Under the regime $p/n\to y>0$, set $V_n=n^{-1} X_n X_n^T$ and consider the sample correlation matrices
\begin{equation*}
R_n= (\diag(V_n))^{-1/2} V_n (\diag(V_n))^{-1/2}\,.
\end{equation*}
It was shown in \cite{heiny:mikosch:2017:corr} that $\E[R_n]=I$. However, from \cite[Theorem 3.1 part (2)]{heiny:mikosch:2017:corr} we know that 
\begin{equation*}
\liminf_{\nto} \E\Big[\int x^k F_{R_n}(dx)\Big]> \beta_k(y)\,,\qquad k\ge 4\,,
\end{equation*}
where $\beta_k(y)$ is the $k$-th moment of the \MP law $\mu^y$. Since $\mu^y$ is uniquely characterized by its moments, the LSD of $R_n$ cannot be $\mu^y$.
\end{example}

\subsection{Our model}
We will consider a data matrix $X_n$ with correlated entries. To this end, we introduce the Curie-Weiss model which is an exactly solvable model of ferromagnetism. ``Because of its
simplicity and because of the correctness of at least of some of its predictions, the Curie-Weiss model occupies an important place in the statistical mechanics literature and its
application to information theory \cite{kochmanski2013curie}.''
The first time that random matrices with Curie-Weiss spins were analyzed was in \autocite{Friesen:zwei}, with subsequent improvements in \cite{hochstaetter:kirsch:warzel:2016,kirsch:kriecherbauer:2018,FKK2019,FleermannCWLL}, where the last two publications are based on \autocite{FleermannDiss}. All of these texts were concerned with Wigner type matrices and convergence to the semicircle distribution.

\begin{definition}\label{def:curieweiss}
Let $n\in\N$ be arbitrary and $Y_1,\ldots,Y_n$ be random variables defined on some probability space $(\Omega,\Acal,\Prob)$. Let $\beta>0$, then we say that $Y_1,\ldots,Y_n$ are Curie-Weiss($\beta$,$n$)-distributed\label{sym:CurieWeissdist}, if for all $y_1,\ldots,y_n\in\{-1,1\}$ we have that
\[
\Prob(Y_1=y_1,\ldots,Y_n=y_n) = \frac{1}{Z_{\beta,n}}\cdot \e^{\frac{\beta}{2n}\left(\sum_{i=1}^n y_i\right)^2}\,,
\]
where $Z_{\beta,n}= \sum_{y_1,\ldots,y_n \in \{-1,1\}} \e^{\frac{\beta}{2n}\left(\sum_{i=1}^n y_i\right)^2}$\label{sym:CWconstant} is a normalization constant. The parameter $\beta$ is called \emph{inverse temperature}.
\end{definition} 

Note that in above definition, $(Y_1,\ldots,Y_n)$ is an exchangeable random vector, since the probability of any spin configuration $(y_1,\ldots,y_n)$ only depends on the sum of the spins. The Curie-Weiss($\beta,n$) distribution is used to model the behavior of $n$ ferromagnetic particles (spins) at the inverse temperature $\beta$. At low temperatures (if $\beta$ is large), all magnetic spins are likely to have the same alignment, resembling a strong magnetic effect. On the contrary, at high temperatures (if $\beta$ is small), spins can act almost independently, resembling a weak magnetic effect. The model exhibits a phase transition at $\beta=1$, meaning that the behavior of the distribution varies significantly in the realms $\beta\in (0,1)$, $\beta=1$ and $\beta>1$. To exemplify a manifestation of this phase transition, we formulate the following result; see Theorem 5.17 in \cite{KirschMomentSurvey}.
\begin{lemma}
\label{lem:CWcorrelations}
Fix $l\in \N$ and let for all  $n\ge l$,  $(Y^{(n)}_1,\ldots,Y^{(n)}_l)$ be part of a Curie-Weiss($\beta,n$) distributed random vector.  If $l$ is even, the following statements hold:
\begin{enumerate}[i)]
\item If $\beta<1$, then for some constant $c(\beta,l)>0$, $\E Y^{(n)}_1\cdots Y^{(n)}_l \sim c(\beta,l) n^{-l/2}$ as $n\to\infty$.
\item If $\beta=1$, then for some constant $c(l)>0$, $\E Y^{(n)}_1\cdots Y^{(n)}_l \sim c(l)n^{-l/4}$ as $n\to\infty$.
\item If $\beta>1$, then $\E Y^{(n)}_1\cdots Y^{(n)}_l \sim m^l$ as $n\to\infty$, where $m\in (0,1)$ is the unique positive number such that $\tanh(\beta m)=m$.
\end{enumerate}
If $l$ is odd, then for all $\beta>0$ one has $\E Y^{(n)}_1\cdots Y^{(n)}_l = 0$.
\end{lemma}
Note that in the setting of Lemma \ref{lem:CWcorrelations}, the correlation $\E Y^{(n)}_1 Y^{(n)}_2$ is of a different order for the three regions of $\beta$.
If $\beta<1$, the correlation $\E Y^{(n)}_1 Y^{(n)}_2$ decays at a rate of $n^{-1}$. For the critical temperature $\beta=1$ the decay rate is $n^{-1/2}$ , whereas if $\beta>1$, the correlation $\E Y^{(n)}_1 Y^{(n)}_2$ converges to $m^2$ and hence does not vanish as $n\to \infty$. In our main result Theorem \ref{thm:CWMP}, we will see that for $\beta>1$  a different normalization of the sample covariance matrix is required to account for the correlation at level $m^2$.

\subsubsection*{Objective and structure of this paper} The aim of this paper is to characterize the LSD of the sample covariance matrices $V_n=n^{-1} X_nX_n^T$, where $X_n$ follows a Curie-Weiss distribution. At the critical temperature $\beta=1$ a phase transition occurs. In Section \ref{sec:mainresult}, we see that the LSD is a possibly rescaled \MP or semicircle distribution. Section \ref{sec:proof} contains some useful lemmas and the proof of our main result.

\subsubsection*{Notation}
For simplicity of notation, we define for all $n\in\N$: $\oneto{n}\defeq\{1,\ldots,n\}$. Further, whenever there is no ambiguity about the dimension we denote the identity matrix by $I$.

\section{Main result}\label{sec:mainresult}

Our main result characterizes the limiting spectral distributions of sample covariance matrices with Curie-Weiss entries with parameter $\beta>0$ in the regimes $p/n\to y>0$ and $p/n\to 0$. 

\begin{theorem}
\label{thm:CWMP}
Assume \eqref{Cgamma} and that the entries of the $p\times n$ matrix $X_n$ are Curie-Weiss$(\beta,np)$ distributed with $\beta>0$,
where we assume that $(X_n(i,j))_{i\in [p], j\in [n]}$ are defined on a common probability space. Denote by $F_n$ the ESD of $V_n\defeq n^{-1}X_nX_n^T$. 
\begin{itemize}
\item[(i)] Assume $\beta\in(0,1]$. If $p/n\to y\in(0,\infty)$, then $(F_n)_n$ converge weakly almost surely to the \MP distribution $\mu^{y}$, as $p\to \infty$. If $p/n\to 0$, then the ESDs of $\sqrt{\tfrac{n}{p}} (V_n- I)$ converge weakly almost surely to the semicircle distribution $G$, as $p\to \infty$.
\item[(ii)] Assume $\beta\in(1,\infty)$ and let $m$ be the unique number in $(0,1)$ satisfying $\tanh (m\beta)=m$. If $p/n\to y\in(0,\infty)$, then the ESDs of $(1-m^2)^{-1} V_n$ converge weakly almost surely to the \MP distribution $\mu^{y}$, as $p\to \infty$. If $p/n\to 0$, then the ESDs of $\sqrt{\frac{n}{p}} \left(\frac{1}{1-m^2}V_n-I\right)$ converge weakly almost surely to the semicircle distribution $G$, as $p\to \infty$.
\end{itemize}
\end{theorem}
By Lemma \ref{lem:CWcorrelations}, the correlations between the entries of $X_n$ increase with the value $\beta$. 
Theorem \ref{thm:CWMP} shows that for $\beta\le 1$  the correlation is still weak enough to not affect the LSD, in the sense that we obtain the same LSD as for a sample of i.i.d.\ random variables. For $\beta>1$ the asymptotic behavior of the correlations changes drastically. Consequently, a different normalization of the sample covariance matrix is required to account for the correlation at constant level $m^2$.
\begin{remark}
\label{rem:npdependence}
The convergence in Theorem \ref{thm:CWMP} is for $p\to \infty$, which is standard in the $p/n\to 0$ literature; see for example \cite{bai:yin:1988,pan2012asymptotic,wang:paul:2014}. If there exists a $\delta>0$ such that $n^\delta/p\to 0$, the convergence also holds for $\nto$; compare also with \cite[Corollary 2]{elkaroui:2009}. Indeed, $n^\delta/p\to 0$ for some $\delta>0$ is equivalent to $p_n^{-a}$ being summable over $n$ for some large $a>0$, which is required for the Borel-Cantelli argument in the proof of Theorem \ref{thm:CWMP}.
Our formulation with $p\to \infty$ is slightly more flexible because it also allows choices such as $p=\log n$. 
\end{remark}
\begin{figure}
   \includegraphics[width=0.5\textwidth]{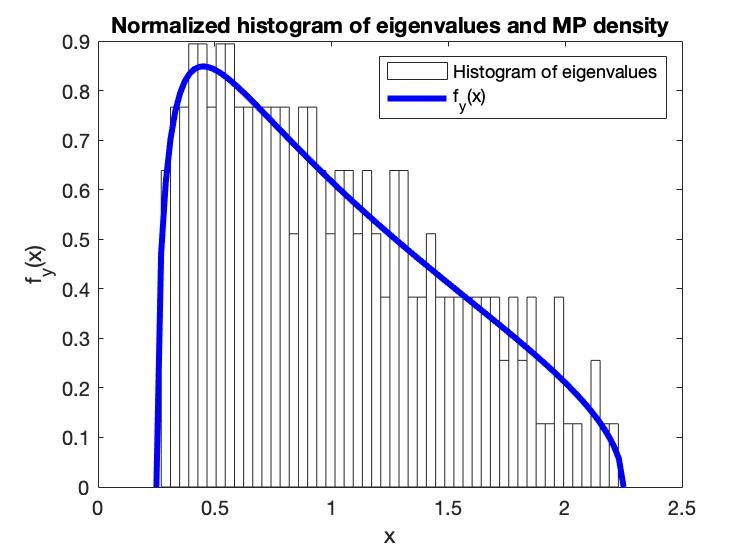}
   \caption{Simulation for $(p,n)=(200,800)$ and $\beta=0.5$. In blue: Density $f_{1/4}$. Histogram: Eigenvalues of $n^{-1}X_nX_n^T$.}
	\label{fig:beta05}
\end{figure}

In Figure 1 and Figure 2, we can see a simulation output where a $200\times 800$ random matrix with Curie-Weiss entries was simulated, using the Metropolis algorithm with $16\cdot 10^6$ steps. We compare the histogram of the eigenvalues with the \MP density $f_{p/n}$,
\begin{equation*}
f_y(x) =
\frac{1}{2\pi xy} \sqrt{((1+\sqrt{y})^2-x)(x-(1-\sqrt{y})^2)} \one_{((1-\sqrt{y})^2,(1+\sqrt{y})^2)}(x)\,, \qquad x\in \R, y\in (0,1]\,. 
\end{equation*}
While in Figure \ref{fig:beta05}, the ensemble was simulated for $\beta = 0.5$, in Figure \ref{fig:beta129} we used $\beta=1.29727$, so that $\tanh(\beta m) = m$ holds for $m=3/4$. The largest eigenvalue of $n^{-1}XX^T$ resp.\ $(n(1-m^2))^{-1}$ was roughly $113$ resp.\ $258$ and was excluded from the histogram in Figure \ref{fig:beta129}.

\begin{figure}
    {\centering
    \begin{minipage}[t]{0.5\linewidth}
        \centering
        \includegraphics[clip=true, trim=0cm 0cm 0cm 0cm, width=\linewidth]{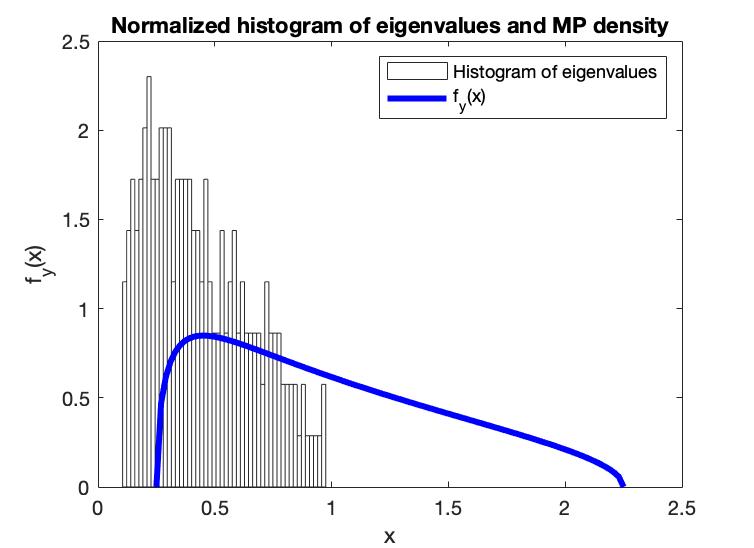}
    \end{minipage}
    \hfill
    \begin{minipage}[t]{0.5\linewidth}
        \centering
        \includegraphics[clip=true, trim=0cm 0cm 0cm 0cm, width=\linewidth]{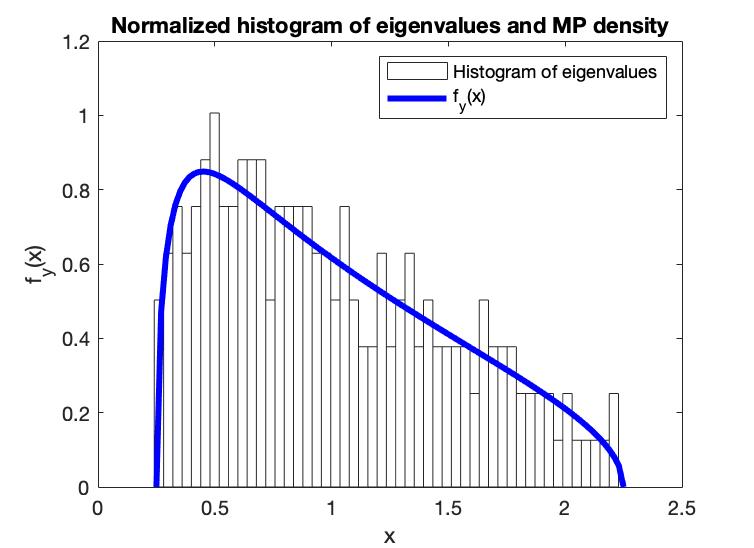}
    \end{minipage}}
    \caption{Simulation for $(p,n)=(200,800)$ and $\beta=1.29727$. In blue: Density $f_{1/4}$. Left histogram: Eigenvalues of $n^{-1}X_nX_n^T$. Right histogram: Eigenvalues of $(n(1-m^2))^{-1} X_nX_n^T$, where $m=3/4$, so that $\tanh(\beta m) =m$.}
\label{fig:beta129}
\end{figure}

In the proof of Theorem~\ref{thm:CWMP}, we use techniques developed in \cite{FleermannCWLL} and \cite{FKKzwei}.
An important tool is the fact that Curie-Weiss($\beta,n$) spins $(Y_1,\ldots,Y_n)$ are \emph{conditionally i.i.d.} That is, without loss of generality we can assume that they are defined on the same probability space as a Lebesgue-continuous mixing variable $M^{\beta}_n$ with support $[-1,1]$, such that conditioned on $M^{\beta}_n=t\in(-1,1)$, $(Y_1,\ldots,Y_n)$ are i.i.d.\ $P_t$-distributed, where $P_t$ is the probability measure on $\{\pm 1\}$ with
\[
P_t(1) = \frac{1+t}{2} \qquad \text{and} \qquad P_t(-1) = \frac{1-t}{2}.
\]
Next, we collect some properties of the mixing variable $M_n^{\beta}$ in the following lemma which is taken from \autocite{KirschMomentSurvey}; see Theorem 5.6, Remark 5.7, Proposition 5.9 and Theorem 5.17 therein.

\begin{lemma}
\label{lem:curiedefinetti}
If $Y=(Y_1,\ldots,Y_n)$ are Curie-Weiss($\beta$,$n$)-distributed for some $\beta>0$ and $n\in\N$, then w.l.o.g.\ there exists a random variable $M^{\beta}_n$ supported on $[-1,1]$ with the following properties.
\begin{enumerate}
\item The distribution of $M^{\beta}_n$ has Lebesgue-density $f^{\beta}_n$, 
\[
f^{\beta}_n(t) \defeq \frac{1}{\int_{(-1,1)} \frac{e^{-\frac{n}{2} F_{\beta}(s)}}{1-s^2} \lebesgue(\text{d}s)}  \frac{e^{-\frac{n}{2} F_{\beta}(t)}}{1-t^2}\one_{(-1,1)}(t)\,, \qquad t\in(-1,1),
\]
where for all $s\in(-1,1)$ we define
\[
F_{\beta}(s) \defeq \frac{1}{\beta} \left(\frac{1}{2}\ln\left(\frac{1+s}{1-s}\right)\right)^2 + \ln(1-s^2).
\]	
\item $\Prob^{M^{\beta}_n}$-almost surely, $\Prob^{Y|M^{\beta}_n=t} = \otimes_{i\in\oneto{n}} \Prob^{Y_i|M^{\beta}_n=t} = \otimes_{i\in\oneto{n}} P_t$. In words, conditionally on $M_n^{\beta}$ the $Y_1,\ldots,Y_n$ are i.i.d.\ $P_t$-distributed random variables.
\item If $\beta< 1$, the mixing variable $M^{\beta}_n$ satisfies the following moment decay:
\[
\forall~a\in 2\N: \int_{[-1,+1]} t^a \Prob^{M^{\beta}_n}(\de t) \leq \frac{K_{\beta,a}}{n^{\frac{a}{2}}}.   
\]
\item If $\beta= 1$, the mixing variable $M^{\beta}_n$ satisfies the following moment decay:
\[
\forall~a\in 2\N: \int_{[-1,+1]} t^a \Prob^{M^{\beta}_n}(\de t) \leq \frac{K_{\beta,a}}{n^{\frac{a}{4}}},   
\]
\end{enumerate}
where $K_{\beta,a}\in \R_+$\label{sym:CWdefinetticonstant} are constants that depend on $\beta$ and $a$ only. 
\end{lemma}
 
In the case $\beta>1$ we will work with suitably restandardized Curie-Weiss spins in order to use the following lemma which can be found in \cite{FKKzwei}.
 
\begin{lemma}
\label{lem:Curiehightemp}
Let $(Y_1,\ldots,Y_n)$ be Curie-Weiss($\beta,n$)-distributed with $\beta>1$ and mixing variable $M^{\beta}_n$. Denote by $m\in(0,1)$ the unique positive number satisfying $\tanh (m\beta)=m$. For $i\in \{1,\ldots,n\}$ define
\[
Z_i \defeq \frac{1}{\sqrt{1-m^2}} \big(Y_i - m\one_{M^{\beta}_n>0} + m\one_{M^{\beta}_n<0}\big).
\]
Then $(Y_1,\ldots,Y_n)$ are conditionally i.i.d.\ given $M^{\beta}_n$ and the following statements hold.
\begin{enumerate}
\item Almost surely, $(Y_1,\ldots,Y_n)$ takes values in $\{\frac{\pm 1 +m}{\sqrt{1-m^2}}\}^n\cup  \{\frac{\pm 1 -m}{\sqrt{1-m^2}}\}	^n$.
\item For each $i\in \{1,\ldots,n\}$, 
\begin{align*}
\E(Z_i|M^{\beta}_n=t) = \zeta(t) &\defeq \begin{cases}
 \frac{1}{\sqrt{1-m^2}}(t-m), & \quad t>0,\\
 \frac{1}{\sqrt{1-m^2}}(t+m), & \quad t<0,
\end{cases}\\
\E(1-Z^2_i|M^{\beta}_n=t) = \psi(t) &\defeq \begin{cases}
 \frac{2m}{1-m^2}(t-m), & \quad t>0,\\
 \frac{2m}{1-m^2}(t+m), & \quad t<0.
\end{cases}
\end{align*}
\item We obtain the following bounds on the moments of $\zeta(M^{\beta}_n)$ and  $\psi(M^{\beta}_n)$:
\begin{align*}
\forall~a\in 2\N: & \int_{[-1,+1]} \abs{\zeta(t)}^a \Prob^{M^{\beta}_n}(\de t) \leq \frac{K_{\beta,a}}{n^{\frac{a}{2}}},  \\
\forall~a\in 2\N: & \int_{[-1,+1]} \abs{\psi(t)}^a \Prob^{M^{\beta}_n}(\de t) \leq \frac{K_{\beta,a}}{n^{\frac{a}{2}}}. 
\end{align*}
\end{enumerate}
Here, the constants $K_{\beta,a}>0$ depend only on $\beta$ and $a$.
\end{lemma}

\section{Proof of Theorem \ref{thm:CWMP}}\label{sec:proof}
We will prove the cases $\beta\leq 1$ and $\beta > 1$ separately, but before we begin, we will provide two lemmas which we will use throughout the proof. The first lemma is taken from \autocite{FleermannCWLL}, see their Theorem 39.

\begin{lemma}
\label{lem:largedev}
Let $n\in \N$ be arbitrary, $(a_{i,j})_{i,j\in\oneto{n}}$ and $(b_i)_{i\in\oneto{n}}$ be deterministic complex numbers, $(Y_i)_{i\in\oneto{n}}$ be independent and complex-valued random variables with common expectation $t\in\C$. Further, we assume that for all $a \geq 2$ there exists a $\mu_a\in\R_+$ such that $\norm{Y_i-t}_a:= \E[|Y_i-t|^{a}]^{1/a} \leq \mu_a$ for all $i\in\oneto{n}$. Then we obtain for all $a\geq 2$:
\begin{align*}
i)&\quad 	 \bignorm{\sum\limits_{i \in \oneto{n}} b_{i} Y_i}_a \leq  \left(A_a \mu_a
+ \sqrt{n}\abs{t} \right) 
\sqrt{\sum\limits_{i \in \oneto{n}}\abs{b_{i}}^2}\,,\\
ii)&\quad
\bignorm{\sum\limits_{\substack{i,j \in \oneto{n}\\i\neq j}} 
a_{i,j} Y_i Y_j}_a \leq 
A_a \mu_a \abs{t} \sqrt{\sum_{j\in\oneto{n}}\bigabs{\sum_{i\in\oneto{n}\backslash\{j\}}a_{i,j}}^2}
+ A_a \mu_a \abs{t} \sqrt{\sum_{i\in\oneto{n}}\bigabs{\sum_{j\in\oneto{n}\backslash\{i\}}a_{i,j}}^2}\\
&\quad+ 
4 A_a^2 \mu^2_a\sqrt{\sum\limits_{\substack{i,j \in \oneto{n}\\i\neq j}}\abs{a_{i,j}}^2} 
 + \abs{t}^2 \left|\sum\limits_{\substack{i,j \in \oneto{n}\\i\neq j}}a_{i,j} \right| 
\leq \left(4 A_a^2 \mu^2_a 
+ 2 A_a \mu_a \sqrt{n}\abs{t}   +
n \abs{t}^2\right) 
\sqrt{\sum\limits_{\substack{i,j \in \oneto{n}\\i\neq j}}\abs{a_{i,j}}^2}\,,
\end{align*}
where $A_a\in \R_+$ is positive constant depending only on $a$.
\end{lemma}
\begin{proof}
See \autocite{FleermannCWLL}.
\end{proof}
The following lemma allows us to apply Lemma \ref{lem:largedev} to the setting we will encounter in our proof.
\begin{lemma}
\label{lem:matrixbounds}
Let $X$ be a $p\times n$ matrix with real-valued entries, $z\in\C_+$. Define
\begin{equation}
\label{eq:Fdefined}
F(X)\defeq X^T\left(\frac{1}{n}XX^T-z\right)^{-1}X.
\end{equation}
Then we obtain the following bounds:
\[
i)\quad\sqrt{\sum_{i\neq j}^n\abs{F_{ij}(X)}^2}\leq n\sqrt{p} \left(1+\frac{\abs{z}}{\Im(z)}\right)\,, \qquad 
ii)\quad\abs{\tr F(X)} \leq np\left(1+\frac{\abs{z}}{\Im(z)}\right).
\]
Further, if all entries in $X$ are uniformly bounded by some $b>0$, we obtain:
\[
iii)\quad\bigabs{\sum_{i,j\in\oneto{n}} F_{ij}(X)} \leq \frac{b^2p}{\Im(z)} + \frac{1}{n}\left(1+ \frac{|z|}{\Im(z)}\right),\qquad iv)\quad \forall\,j\in\oneto{n} : \bigabs{\sum_{i\in\oneto{n}\backslash\{j\}} F_{ij}(X)} \leq \frac{b^2 p}{n\Im(z)}.
\]
\end{lemma}
\begin{proof}
To prove i), we recall that
\begin{enumerate}[a)]
\item $\text{Spectrum}(X^T(XX^T-z)^{-1}X)\cup\{0\} = \text{Spectrum}((XX^T-z)^{-1}XX^T)\cup\{0\}$,
\item $(XX^T-z)^{-1}XX^T= I + z (XX^T-z)^{-1}$,
\end{enumerate}
and that $\norm{\cdot}_F\leq\sqrt{m}\norm{\cdot}$ for $m\times m$ matrices, where $\norm{\cdot}_F$ denotes the Frobenius norm and $\norm{\cdot}$ denotes the operator norm. Therefore,

\begin{align*}
\sqrt{\sum_{i\neq j}^n\abs{F_{ij}(X)}^2}&\leq \sqrt{ \sum_{i\neq j}^n \bigabs{\left[X^T\left(\frac{1}{n}XX^T-z\right)^{-1} X\right](i,j)}^2} 
= n \bignorm{\frac{1}{n}X^T\left(\frac{1}{n}XX^T-z\right)^{-1} X}_F\notag\\
&= n \bignorm{\left(\frac{1}{n}XX^T-z\right)^{-1} \left(\frac{1}{n}XX^T\right)}_F
= n \bignorm{I_{p-1} + z\left(\frac{1}{n}XX^T-z\right)^{-1} }_F\notag\\
&\leq n\sqrt{p} \bignorm{I_{p-1} + z\left(\frac{1}{n}XX^T-z\right)^{-1}}
\leq n\sqrt{p} \left(1+\frac{|z|}{\Im(z)}\right).
\end{align*}

For ii) we calculate 
\begin{equation*}\label{eq:bound1}
\begin{split}
&\abs{\tr F(X)}= n \bigabs{\tr\left(\left(\frac{1}{n}XX^T-z\right)^{-1} \left(\frac{1}{n}XX_n^T\right) \right)}\\
&\leq n p \bignorm{\left(\frac{1}{n}XX^T-z\right)^{-1} \left(\frac{1}{n}XX^T\right)} \leq np \left(1+\frac{\abs{z}}{\Im(z)}\right),
\end{split}
\end{equation*}
where the last step follows as in the proof of i). This shows ii). For iii) let $1_n\defeq(1,\ldots,1)^T\in \R^n$ and $Y\defeq n^{-1/2}X$. Then we see that 
\begin{equation*}\label{eq:bound2}
\begin{split}
n &\bigabs{ \sum_{i,j\in\oneto{n}} F_{ij}(X)} = \bigabs{1_n^T Y^T(YY^T-z)^{-1}Y1_n}\\
&\leq \bigabs{1_n^T Y^T \left[(YY^T-z)^{-1} - (YY^T +Y1_n1_n^TY^T-z)^{-1}\right]Y1_n}\\
&\quad+ \bigabs{1_n^T Y^T  (YY^T +Y1_n1_n^TY^T-z)^{-1} Y1_n} =:P_1+P_2\,.
\end{split}
\end{equation*}
By \cite[Lemma~2.6]{silverstein:bai:1995}, one has
\begin{equation*}
P_1\leq \frac{\norm{Y1_n1_n^T Y^T}}{\Im(z)}=\frac{\abs{1_n^T Y^TY1_n}}{\Im(z)}= \frac{1}{\Im(z)}\bigabs{\frac{1}{n} \sum_{i\in\oneto{n}}\sum_{j\in\oneto{p}}\sum_{s\in\oneto{n}} X_{ij}X_{js}} \leq \frac{b^2 np}{\Im(z)}\,.
\end{equation*}
To bound $P_2$, recall from \cite{yaskov:2016} that for a real, symmetric, positive semidefinite $m\times m$ matrix $M$; $x\in \R^m$, $z\in \C_+$ the following inequality holds:
\begin{equation}\label{eq:le3}
\bigabs{ x^T (M+xx^T-z)^{-1}x}\leq 1+ \frac{|z|}{\Im(z)}\,.
\end{equation}
So in particular $P_2$ is bounded by the r.h.s.\ of \eqref{eq:le3}. This yields the bound
\[
\bigabs{ \sum_{i, j\in\oneto{n}} F_{ij}(X)} \leq \frac{P_1 + P_2}{n} \leq \frac{b^2p}{\Im(z)} + \frac{1}{n}\left(1+ \frac{|z|}{\Im(z)}\right).
\]
Lastly, to show $iv)$ let $Y$ be defined as before and let $j\in\oneto{n}$ be arbitrary. Denote by $y$ the $j$-th standard basis vector of $\R^n$ and let $x\defeq 1_n-y$. Then, using that $Yy x^T Y^T$ has rank one,
\begin{align*}
&n\bigabs{\sum_{i\in\oneto{n}\backslash\{j\}} F_{ij}(X)}= \bigabs{x^T Y^T(YY^T-z)^{-1}Yy}  = \bigabs{\tr \left[(YY^T-z)^{-1}Yy x^T Y^T\right]}\\
&\leq \bignorm{(YY^T-z)^{-1}Yy x^T Y^T} \leq  \bignorm{(YY^T-z)^{-1}} \bignorm{Yy x^T Y^T}\\
&\leq \frac{1}{\Im(z)}\bigabs{x^T Y^TYy} = \frac{1}{\Im(z)}\bigabs{\frac{1}{n}\sum_{i\in\oneto{n}\backslash \{j\}} \sum_{s\in\oneto{p}} X_{js}X_{is}}\leq \frac{b^2 p}{\Im(z)}.
\end{align*}
\end{proof}

\subsection{The case $\beta\le 1$.}
 We will show that the Stieltjes transforms converge.  For a real symmetric matrix $M\in \R^{p\times p}$ we denote by $S_M$ the Stieltjes transform of $M$, that is:
\[
\forall\, z\in\C_+: S_M(z) = \frac{1}{p}\tr(M-z)^{-1}.
\]
Further, we write $s_n$ for the Stieltjes transform of $V_n$.

 
Fix a $z\in\C_+$. Our starting point is the following identity, which is easy to verify: 

\begin{equation}\label{eq:stgser}
S_{y_n^{-1/2}(V_n-I)}(z)= y_n^{1/2} s_n(1+y_n^{1/2} z),\, \quad \text{where } y_n \defeq\frac{p}{n}.
\end{equation}
For simplicity of notation we write $\eta= \Im(z)>0$ and  $q= q_n = 1+y_n^{1/2} z$. Note that $\Im(q)=\eta\sqrt{p/n}$.
We know by equation (3.3.6) in \cite{bai:silverstein:2010} that 
\begin{align}
s_n(q) &= \frac{1}{p}\sum_{k\in\oneto{p}} \frac{1}{\frac{1}{n}\alpha_k^T\alpha_k-q-\frac{1}{n^2}\alpha_k^T X_n^{(k)}\left(\frac{1}{n}X_n^{(k)}X_n^{(k)T}-q\right)^{-1}X_n^{(k)}\alpha_k} \notag \\
&= \frac{1}{1 - q - y_n - y_n q s_n(q)} - \delta_n(q) \label{eq:sneq},
\end{align}
where $\alpha_k^T$ is the $k$-th row of $X_n$ (note that $\alpha_k$ also depends on $n$, which we drop from the notation), $X_n^{(k)}$ is $X_n$ with its $k$-th row removed (thus a $(p-1)\times n$-matrix). Further, 
\begin{equation}\label{eq:delta}
\delta_n(q) = \frac{1}{p} \sum_{k\in\oneto{p}} \frac{\Omega_k^{(n)}(q)}{(1-q-y_n-y_nqs_n(q))(1-q-y_n-y_n q s_n(q)+\Omega^{(n)}_k(q))},
\end{equation}
where for all $k\in\{1,\ldots,p\}$:
\begin{align*}
\Omega_k^{(n)}(q) &= \underbrace{\frac{1}{n}\alpha_k^T\alpha_k - 1}_{=0} - \frac{1}{n^2}\alpha_k^T X_n^{(k)T}\left(\frac{1}{n}X_n^{(k)}X_n^{(k)T}-q\right)^{-1}X_n^{(k)}\alpha_k + y_n + y_n q s_n(q).
\end{align*}
Solving \eqref{eq:sneq}, we obtain analogously to \cite[pp.~55 and 56]{bai:silverstein:2010} that
\begin{equation}\label{eq:sols1}
s_{n}(q)= \frac{1}{2 y_n q} \left(1-q-y_n-y_n q \delta_n(q)+\sqrt{(1-q-y_n+y_n q \delta_n(q))^2-4 y_n q}\right)\,.
\end{equation}
If $y_n\to y>0$, we see  from \eqref{eq:sols1} that $s_n(q)$ converges almost surely to $S_{\mu^y}(1+\sqrt{y}z)$ provided $\delta_n(q)\to 0$ almost surely as $\nto$. Here $S_{\mu^y}$ is the Stieltjes transform of the \MP law $\mu^y$. Then also $s_n(1+\sqrt{y}z)\to S_{\mu^{y}}(1+\sqrt{y}z)$ almost surely for all $z\in\C_+$, since all $s_n$ are $(\min_n\Im(\sqrt{y_n}z))^{-2}<\infty$ Lipschitz on the relevant domain. Therefore, $\mu_n\to\mu^y$ weakly almost surely.

If $y_n\to 0$, a straightforward calculation using \eqref{eq:sols1} and the definition of $q$ yields as $p\to \infty$,
\begin{equation*}
y_n^{1/2} \,s_{n}(q)
= \frac{-z+y_n^{1/2} \delta_n(q) q +\sqrt{z^2-4+2z y_n^{1/2}\delta_n(q) +2 y_n \delta_n(q) +y_n\delta_n^2(q) }}{2}+o(1)\,.
\end{equation*}
We see that $S_{y_n^{-1/2}(V_n-I)}(z)= y_n^{1/2} s_n(q)$ converges almost surely to the Stieltjes transform $s_G(z)$ of the semicircle law provided 
\begin{equation}\label{eq:grsgrsd}
\lim_{p \to \infty} y_n^{1/2} \delta_n(q) =0\quad \as
\end{equation}
Thus, condition \eqref{eq:grsgrsd} suffices for both cases $p/n\to y>0$ and $p/n \to 0$.
It remains to prove \eqref{eq:grsgrsd}.

\subsection{Proof of \eqref{eq:grsgrsd}}

Recall the definition of $\delta_n(q)$ in \eqref{eq:delta}. First, we lower bound the denominator.
By (3.3.13) in \cite{bai:silverstein:2010} and p.~57 below (3.3.15), we have
\begin{equation}\label{eq:dsgsdss}
\begin{split}
\Im(1-q-y_n -y_n q s_{n}(q))&\le - \Im (q)\,,\\
\Im(1-q-y_n -y_n q s_{n}(q)+\Omega^{(n)}_k(q))&\le - \Im (q)\,.
\end{split}
\end{equation}
Using \eqref{eq:dsgsdss} we see that
\begin{equation*}\label{eq:sufficesaverage}
\abs{y_n^{1/2} \delta_n(q)} \leq y_n^{1/2} \abs{\Im(q)}^{-2} \frac{1}{p}\sum_{k\in\oneto{p}}\abs{\Omega_k^{(n)}(q)} = \eta^{-2} y_n^{-1/2}\frac{1}{p}\sum_{k\in\oneto{p}}\abs{\Omega_k^{(n)}(q)}\,.
\end{equation*}
In particular, it suffices to show that 
\begin{equation}\label{eq:segtgts}
\lim_{p \to \infty} y_n^{-1/2} \max_{k=1,\ldots,p}\abs{\Omega_k^{(n)}(q)}=  0 \quad \as 
\end{equation}
Now we prove \eqref{eq:segtgts}. Note that 
\begin{align*}
&\Omega_k^{(n)}(q)
= - \frac{1}{n^2}\alpha_k^T X_n^{(k)T}\left(\frac{1}{n}X_n^{(k)}X_n^{(k)T}-q\right)^{-1}X_n^{(k)}\alpha_k + y_n + y_n z s_n(z)\notag\\
&=  \left(- \frac{1}{n^2} \sum_{i\neq j}^n \alpha_k(i) \left[ X_n^{(k)T}\left(\frac{1}{n}X_n^{(k)}X_n^{(k)T}-q\right)^{-1} X_n^{(k)}\right](i,j) \alpha_k(j) \right)\\
&\quad\, + \left( - \frac{1}{n^2} \tr X_n^{(k)T}\left(\frac{1}{n}X_n^{(k)}X_n^{(k)T}-q\right)^{-1} X_n^{(k)} + y_n + y_n q s_n(z)\right)\\
&=: A(n,k,q) + B(n,k,q).
\end{align*}
We analyse $B(n,k,q)$ first. We have
\begin{align*}
-\frac{1}{n^2}\tr X_n^{(k)T}\left(\frac{1}{n}X_n^{(k)}X_n^{(k)T}-q\right)^{-1}X_n^{(k)}
&= -\frac{1}{n} \tr\left[I_{p-1} + q \left(\frac{1}{n}X_n^{(k)}X_n^{(k)T}-q\right)^{-1}\right]\\
&=-\frac{p}{n} + \frac{1}{n} - \frac{q}{n}\tr\left(\frac{1}{n}X_n^{(k)}X_n^{(k)T}-q\right)^{-1}\,.
\end{align*}
Hence, using $y_n=p/n$, $\Im(q)=\sqrt{y_n}\eta$, and (A.1.12) in \cite{bai:silverstein:2010}, we find
\begin{align*}
\abs{B(n,k,q)} &= \bigabs{ -\frac{p}{n} + \frac{1}{n} - \frac{q}{n}\tr\left(\frac{1}{n}X_n^{(k)}X_n^{(k)T}-q\right)^{-1} + y_n + y_n q\frac{1}{p}\tr\left(\frac{1}{n}X_nX_n^T-q\right)^{-1}}\\
&\leq \frac{1}{n} + \frac{\abs{q}}{n\Im(q)} = \frac{1}{n} + \frac{\abs{1+\sqrt{y_n}z}}{n\sqrt{y_n}\eta}.
\end{align*}
Since this bound holds uniformly for all $k\in\{1,\ldots,p\}$, it follows that 
\[
y_n^{-1/2}\max_{k=1,\ldots,p}\abs{B(n,k,q)} \leq \frac{1}{n \sqrt{p/n}} + \frac{\abs{1+z\sqrt{p/n}}}{p\eta} \xrightarrow[p\to\infty]{} 0 \quad \as
\]
It is left to show that $y_n^{-1/2}\max_{k=1,\ldots,p}\abs{A(n,k,q)} \to 0$ almost surely. We do so by bounding 
the terms 
\[
S(n,k,q)\defeq n^2 A(n,k,q) = \sum_{i\neq j}^n \alpha_k(i) \left[X_n^{(k)T}\left(\frac{1}{n}X_n^{(k)}X_n^{(k)T}-q\right)^{-1} X_n^{(k)}\right](i,j) \alpha_k(j)	
\]
using Lemma~\ref{lem:largedev} and Lemma~\ref{lem:matrixbounds}. In accordance with Lemma~\ref{lem:matrixbounds}, we consider the \emph{symmetric} matrix
\[
F\left(X_n^{(k)}\right) = X_n^{(k)T}\left(\frac{1}{n}X_n^{(k)}X_n^{(k)T}-q\right)^{-1} X_n^{(k)}.
\]
 Using $q=1+\sqrt{y_n}z$, we draw the following corollary of Lemma~\ref{lem:matrixbounds}:
\begin{corollary}\label{cor:SecondFbounds}
For any $a\in\N$ there exists a constant $C_a>0$ independent of $n$ and $p$ such that for any $k\in\oneto{p}$ and any realization $X$ of $X^{(k)}_n$, we find
\begin{enumerate}[i)]
\item $\left(\sum_{i\neq j}\abs{F_{ij}(X)}^2\right)^{\frac{a}{2}} \leq C_a n^{3a/2}$
\item $\bigabs{\sum_{i\neq j} F_{ij}(X)}^a \leq C_a n^{3a/2}p^{a/2}$
\item $\left(\sum_{j\in\oneto{n}}\bigabs{\sum_{i\in\oneto{n}\backslash\{j\}} F_{ij}(X)}^2\right)^{\frac{a}{2}} \leq C_a p^{a/2}$.
\end{enumerate}

\end{corollary}
\begin{proof}
We will use Lemma~\ref{lem:matrixbounds} throughout the proof. For $i)$ we obtain
\[
\left(\sum_{i\neq j}\abs{F_{ij}(X)}^2\right)^{\frac{1}{2}} \leq n\sqrt{p-1}\left(1+\frac{\abs{q}}{\eta}\sqrt{\frac{n}{p}}\right) \leq C_1 n^{3/2}
\]
for some constant $C_1$ independent of $n$ and $p$. For ii) we note that
\begin{align*}
&\bigabs{\sum_{i\neq j} F_{ij}(X)} \leq 
\bigabs{\sum_{i,j} F_{ij}(X)} + \abs{\tr F(X)}\leq \frac{p}{\Im(q)} + \frac{1}{n}\left(1+\frac{\abs{q}}{\Im(q)}\right) + np\left(1+\frac{\abs{q}}{\Im(q)}\right)\\
&\leq\frac{\sqrt{pn}}{\eta} + \frac{1}{n} + \frac{\abs{q}}{\eta\sqrt{pn}} + np + \frac{\abs{q} n^{3/2}p^{1/2}}{\eta} \leq C_2  n^{3/2}p^{1/2}
\end{align*}
for some constant $C_2$ independent of $n$ and $p$. Now for $iii)$ we calculate
\[
\left(\sum_{j\in\oneto{n}}\bigabs{\sum_{i\in\oneto{n}\backslash\{j\}} F_{ij}(X)}^2\right)^{\frac{1}{2}} \leq \left(n \frac{p}{n\eta^2}\right)^{\frac{1}{2}}\leq C_3 p^{1/2}
\]  
for a constant $C_3$ independent of $n$ and $p$.
This shows the statement with $C_a\defeq\max\{C_1^a,C_2^a,C_3^a\}$.
\end{proof}

Throughout this section the random variable $M_{np}^{\beta}$ satisifies the properties listed in Lemma \ref{lem:curiedefinetti} if $\beta\le 1$ or those in Lemma \ref{lem:Curiehightemp} if $\beta>1$.

Note that $X_n$ is a matrix made up of $np$ Curie-Weiss($\beta$, $np$) spins, and that for any $k\in\oneto{p}$, $\alpha_k$ is the $k$-th row of $X_n$ and thus contains variables disjoint from the variables in $X_n^{(k)}$.
In what follows, we will use that for $r,s,t\geq 0$ and $a\in\N$ we have $(s+t)^a \leq 2^a(s^a+t^a)$ and $(r+s+t)^a \leq 4^a (r^a + s^a + t^a)$.
We calculate for $a\in 2\N$ and $k\in\oneto{p}$ arbitrary (where sums over $i\neq j$ are for $i,j\in \oneto{n}$, and further explanations can be found beneath the calculation):
\begin{align*}
 &\E\abs{S(n,k,q)}^a = \E\E[\abs{S(n,k,q)}^a|M^{\beta}_{np}] = \E \E\left[\sum_{i\neq j}\alpha_k(i)F_{ij}\left(X^{(k)}_n\right)\alpha_k(j)\,\middle|\, M^{\beta}_{np} \right] \\
& =\int_{[-1,1]}\int_{\{\pm 1\}^{(p-1)\times n}} \int_{\{\pm 1\}^n} \bigabs{\sum_{i\neq j} x_i F_{ij}(X) x_j}^a \text{d} \Prob^{\alpha_k|M^{\beta}_{np}=t}(x)\text{d} \Prob^{X^{(k)}_n|M^{\beta}_{np}=t}(X)\text{d}\Prob^{M^{\beta}_{np}}(\de t)\\
& \leq \int_{[-1,1]}\int_{\{\pm 1\}^{(p-1)\times n}}  
4^a (4 A_a^2 \mu^2_a)^a \left(\sum_{i\neq j}\abs{F_{ij}(X)}^2\right)^{\frac{a}{2}}
\text{d} \Prob^{X^{(k)}_n|M^{\beta}_{np}=t}(X)\text{d}\Prob^{M^{\beta}_{np}}(\de t)\\
& \quad + \int_{[-1,1]}\int_{\{\pm 1\}^{(p-1)\times n}}  
4^a (2 A_a \mu_a)^a \abs{t}^a \left(
\sum_{j\in\oneto{n}}\bigabs{\sum_{i\in\oneto{n}\backslash\{j\}} F_{ij}(X)}^2\right)^{\frac{a}{2}}
\text{d} \Prob^{X^{(k)}_n|M^{\beta}_{np}=t}(X)\text{d}\Prob^{M^{\beta}_{np}}(\de t)\\
&\quad + \int_{[-1,1]}\int_{\{\pm 1\}^{(p-1)\times n}}  
4^a\abs{t}^{2a} \bigabs{\sum_{i\neq j} F_{ij}(X)}^a
\text{d} \Prob^{X^{(k)}_n|M^{\beta}_{np}=t}(X)\text{d}\Prob^{M^{\beta}_{np}}(\de t)\\
&\leq K\left( n^{3a/2} + p^{a/2} \int_{[-1,1]}\abs{t}^a\de\Prob^{M^{\beta}_{np}}(\de t) + n^{3a/2}p^{a/2}\int_{[-1,1]}\abs{t}^{2a}\de\Prob^{M^{\beta}_{np}}(\de t)\right)\\ 
&\leq K \left(n^{3a/2} + p^{a/2} (np)^{-a/4} + n^{3a/2}p^{a/2}(np)^{-a/2}\right)\leq K n^{3a/2}. 
\end{align*}
where for the fourth step we used Lemma~\ref{lem:largedev} and the constants $A_a$ and $\mu_a$ therein (note that $F(X)$ is symmetric), in the fifth step we used Corollary~\ref{cor:SecondFbounds} and from here on out, $K$ denotes a constant not depending on $n$ and $p$, but only on $a$,$\beta$ and $\eta$, and $K$ may change its value from one occurrence to the next. In the sixth step we applied Lemma~\ref{lem:curiedefinetti}. Hence, if $\epsilon > 0$ is arbitrary, we calculate
\[
\Prob\left(\frac{\max_{k\in\oneto{p}}\abs{A(n,k,q)}}{\sqrt{y_n}} > \epsilon\right)\\
\leq  p \,\max_{k\in\oneto{p}} \Prob(
S(n,k,q)^a > \epsilon n^{3a/2}p^{a/2})\leq \frac{ pKn^{3a/2}}{\epsilon^a n^{3a/2}p^{a/2}}\,,
\]
which is summable in $p$ for (say) $a=6$.

\subsection{The case $\beta>1$.} To prove part (iii) of Theorem~\ref{thm:CWMP}, let $\beta >1$. Instead of the matrix $X_n$ we consider
\[
\tilde{X}_n \defeq \frac{1}{\sqrt{1-m^2}}\left(X_n(i,j) - m \one_{M^{\beta}_{np}>0}+m\one_{M^{\beta}_{np}<0}\right)_{ij},
\] 
which for every realization is just a rank 1 perturbation of $(1-m^2)^{-1/2}X_n$. As a consequence, it suffices to prove Theorem~\ref{thm:CWMP} (iii) for $ \tilde{V}_n \defeq n^{-1} \tilde{X}_n\tilde{X}_n^T$ instead of  $ (1-m^2)^{-1}V_n$. Using the terminology as above, but substituting $\tilde{X}_n$ for $X_n$ and $\tilde{V}_n$ for $V_n$, we obtain new terms $\tilde{s}_n$, $\tilde{\Omega}^{(k)}_n$, $\tilde{\alpha}_k$, $\tilde{\delta}_n$ and $\tilde{S}(n,k,z)$.  Inspecting above proof for the case $\beta\leq 1$ and observing \eqref{eq:segtgts}, it will suffice to show
\begin{equation*}\label{eq:tildesegtgts}
\lim_{p\to \infty} y_n^{-1/2} \max_{k=1,\ldots,p}\abs{\tilde{\Omega}_k^{(n)}(q)}=  0 \quad \as 
\end{equation*}
Here,
\begin{align*}
\tilde{\Omega}_k^{(n)}(q) &= - \frac{1}{n^2}\tilde{\alpha}_k^T \tilde{X}_n^{(k)T}\left(\frac{1}{n}\tilde{X}_n^{(k)}\tilde{X}_n^{(k)T}-q\right)^{-1}\tilde{X}_n^{(k)}\tilde{\alpha}_k + y_n + y_n q \tilde{s}_n(q) + \frac{1}{n}\tilde{\alpha}_k^T\tilde{\alpha}_k - 1\\
&=  \left(- \frac{1}{n^2} \sum_{i\neq j}^n \tilde{\alpha}_k(i) \left[ \tilde{X}_n^{(k)T}\left(\frac{1}{n}\tilde{X}_n^{(k)}\tilde{X}_n^{(k)T}-q\right)^{-1} \tilde{X}_n^{(k)}\right](i,j) \tilde{\alpha}_k(j) \right) \\
&\quad\,  +\left(- \frac{1}{n^2} \tr \tilde{X}_n^{(k)T}\left(\frac{1}{n}\tilde{X}_n^{(k)}\tilde{X}_n^{(k)T}-q\right)^{-1} \tilde{X}_n^{(k)} + y_n + y_n q \tilde{s}_n(z)\right)\\
&\quad\, + \left(\frac{1}{n}\tilde{\alpha}_k^T\tilde{\alpha}_k - 1\right)\\
&=: \tilde{A}(n,k,q) + \tilde{B}(n,k,q) + \tilde{C}(n,k,q).
\end{align*}
The term $\tilde{B}(n,k,q)$ can be treated analogously to the term $B(n,k,q)$ above, so we obtain
\[
y_n^{-1/2}\max_{k=1,\ldots,p}\abs{\tilde{B}(n,k,q)} \leq \frac{1}{n \sqrt{p/n}} + \frac{\abs{1+z\sqrt{p/n}}}{p\eta} \xrightarrow{p\to\infty} 0 \quad \as
\]

To handle $\tilde{A}(n,k,q)$, we use Lemma~\ref{lem:Curiehightemp} and the definitions therein. $\tilde{X}_n$ is a matrix made up of $np$ perturbed Curie-Weiss($\beta,np$) spins. Now $\tilde{\alpha}_k$ is the $k$-th row of $\tilde{X}_n$ and thus contains variables disjoint from those in $\tilde{X}_n^{(k)}$. Analogously to the case above we consider the term
\[
F(\tilde{X}_n^{(k)}) = \tilde{X}_n^{(k)T}\left(\frac{1}{n}\tilde{X}_n^{(k)}\tilde{X}_n^{(k)T}-q\right)^{-1} \tilde{X}_n^{(k)}.
\]
We use a slightly faster calculation than for the case $\beta\leq 1$, where a finer analysis was necessary due to the slowly decaying correlations when $\beta=1$. In the following, we will directly compare $\tilde{S}(n,k,q)$ to 
\[
\tilde{R}(n,k,q) \defeq \left(\sum_{i\neq j}\abs{F_{ij}(\tilde{X}^{(k)}_n)}^2\right)^{\frac{1}{2}}.
\]
Note that $\tilde{R}(n,k,q)$ never vanishes, so we my divide by it. Now
 for $T>0$ and $a\in 2\N$ (and where sums over $i\neq j$ are for $i,j\in \oneto{n}$) we calculate for $k\in\oneto{p}$:
\begin{align*}
&\Prob\left(
\tilde{S}(n,k,q) > T \tilde{R}(n,k,q)\right)\\
&=\Prob\left(\bigabs{\sum_{i\neq j} \tilde{\alpha}_k(i) F_{ij}(\tilde{X}^{(k)}_n) \tilde{\alpha}_k(j)} > T \left(\sum_{i\neq j}\abs{F_{ij}(\tilde{X}^{(k)}_n)}^2\right)^{\frac{1}{2}}\right)\\
&\leq\frac{1}{T^a} \int_{[-1,1]}\int_{\Zcal_2} \int_{\Zcal_1} \bigabs{\frac{\sum_{i\neq j} \tilde{x}_i F_{ij}(\tilde{X}) \tilde{x}_j}{\left(\sum_{i\neq j}\abs{F_{ij}[\tilde{X}]}^2\right)^{\frac{1}{2}}}}^a \text{d} \Prob^{\tilde{\alpha}_k|M^{\beta}_{np}=t}(\tilde{x})\text{d} \Prob^{\tilde{X}^{(k)}_n|M^{\beta}_{np}=t}(\tilde{X})\text{d}\Prob^{M^{\beta}_{np}}(\de t)\\
& \leq\frac{1}{T^a} \int_{[-1,1]}\int_{\Zcal_2}  
\left[4 A_a^2 \mu^2_a 
+ 2 A_a \mu_a \sqrt{n}\abs{\zeta(t)}   +
n \abs{\zeta(t)}^2 \right]^a
\text{d} \Prob^{\tilde{X}^{(k)}_n|M^{\beta}_{np}=t}(y_K)\text{d}\Prob^{M^{\beta}_{np}}(\de t)\\
& \leq\frac{1}{T^a} \int_{[-1,1]} 
4^a (4 A_a^2 \mu^2_a)^a \text{d}\Prob^{M^{\beta}_{np}}(\de t)
~+~ \frac{1}{T^a} \int_{[-1,1]}
4^a (2 A_a \mu_a)^a n^{a/2}\abs{\zeta(t)}^a
\text{d}\Prob^{M^{\beta}_{np}}(\de t)\\
& \quad +\frac{1}{T^a} \int_{[-1,1]} 
n^a \abs{\zeta(t)}^{2a}
\text{d}\Prob^{M^{\beta}_{np}}(\de t)\\
&\leq \frac{1}{T^a}\left(K + K n^{a/2} \int_{[-1,1]}\abs{\zeta(t)}^a
\text{d}\Prob^{M^{\beta}_{np}}(\de t) + K n^a \int_{[-1,1]}
\abs{\zeta(t)}^{2a}
\text{d}\Prob^{M^{\beta}_{np}}(\de t)\right)\\
&\leq \frac{K}{T^a},
\end{align*}
where $\Zcal_1$ and $\Zcal_2$ denote the ranges of $\tilde{\alpha}_k$ and $\tilde{X}^{(k)}_n$, respectively (cf.\ Lemma~\ref{lem:Curiehightemp}). Further, in the third step we used Lemma~\ref{lem:largedev} and in the last step  Lemma~\ref{lem:Curiehightemp}. Again, $K$ denotes a floating constant which may change its value from one occurrence to the next, but remains independent of $k$, $n$ and $p$. This helps, since now
\begin{align*}
&\frac{K}{T^a} 
\geq\Prob\left(\tilde{S}(n,k,q) > T \tilde{R}(n,k,q)\right) 
\geq\Prob\left(n^2\tilde{A}(n,k,q) > T n\sqrt{p}\left(1+\frac{\abs{q}}{\Im(q)}\right)\right)\\
&
=\Prob\left(\frac{\tilde{A}(n,k,q)}{\sqrt{p/n}} >  \frac{T}{\sqrt{n}}\left(1+\frac{\abs{q}\sqrt{n}}{\eta\sqrt{p}}\right)\right)
=\Prob\left(\frac{\tilde{A}(n,k,q)}{\sqrt{y_n}} >  \frac{T}{\sqrt{n}}+\frac{T\abs{q}}{\eta\sqrt{p}}\right),
\end{align*}
where in the second step we used the bound on $\tilde{R}(n,k,q)$ given by Lemma~\ref{lem:matrixbounds}. Choosing $T=p^{1/4}$, $a\in 2\N$ such that $a>8$ and using the union bound will show by Borel-Cantelli that
\[
y_n^{-1/2}\max_{k=1,\ldots,p}\abs{\tilde{A}(n,k,q)} \xrightarrow[p\to\infty]{} 0 \quad \text{almost surely.}
\]
It is left to analyze $\tilde{C}(n,k,q)$. Note that this term vanished in the case $\beta\leq 1$. Note also that it suffices to show
\[
\frac{1}{\sqrt{y_n}}\max_{k\in\oneto{p}}\bigabs{\sum_{l=1}^n (\tilde{X}(k,l)^2-1)} = \frac{1}{\sqrt{y_n}}\max_{k\in\oneto{p}} \abs{\tilde{C}(n,k,q)} \xrightarrow[p\to\infty]{} 0 \qquad \text{a.s.}
\]
To this end, for $T>0$, $k\in\oneto{p}$ and $a\in 2\N$ arbitrary it holds (with explanations below)
\begin{align*}
&\Prob\left(\bigabs{\frac{1}{n}\sum_{l=1}^n (\tilde{X}_n(k,l)^2-1)} > T \right)\\
&\leq \frac{1}{(Tn)^a} \E \bigabs{\sum_{l=1}^n (\tilde{X}_n(k,l)^2-1)}^a\\
&= \frac{1}{(Tn)^a} \int_{[-1,1]}\int_{\Zcal_1}\bigabs{\sum_{l=1}^n (x_{kl}^2 -1)}^a \de \Prob^{\tilde{\alpha}_k|M_{np}^{\beta}=t}(x)\de\Prob^{M_{np}^{\beta}}(\de t)\\
&\leq \frac{1}{(Tn)^a} \int_{[-1,1]} \left[(A_a\mu_a + \sqrt{n}\psi(t))\sqrt{n}\right]^a \de\Prob^{M^{\beta}_{np}}(\de t)\\
&= \frac{1}{(Tn)^a} \int_{[-1,1]} \left[A_a\mu_a\sqrt{n} + n\psi(t)\right]^a \de\Prob^{M^{\beta}_{np}}(\de t)\\
& \leq\frac{1}{(Tn)^a}\left(2^a(A_a\mu_a)^a n^{a/2} + 2^a n^a \int_{[-1,1]}\psi(t)^a \de\Prob^{M^{\beta}}_{np}(\de t) \right)\\
&\leq \frac{1}{(Tn)^a} K n^{a/2} + \frac{1}{(Tn)^a}K \frac{n^a}{n^{a/2}}
\leq \frac{K}{T^an^{a/2}}.
\end{align*}
where in the first step we used Markov's inequality, in the second step conditional expectations, in the third step Lemma~\ref{lem:largedev} i) with $b_i=1$, and in the last step Lemma~\ref{lem:Curiehightemp}.

Choosing $\epsilon>0$ arbitrarily and setting $T\defeq \epsilon\sqrt{p/n}$ we obtain for $a\in\N$ with $a\in 2\N$ arbitrary that
\[
\Prob\left(\frac{1}{\sqrt{y_n}}\max_{k\in\oneto{p}} \abs{\tilde{C}(n,k)} > \epsilon \right) \leq 
\max_{k\in\oneto{p}} \Prob\left(\abs{\tilde{C}(n,k)} > \epsilon\sqrt{y_n} \right)\leq 
 \frac{pK}{\epsilon^a p^{a/2}},
\]
which is summable over $p$ for $a> 4$. This ends the proof via Borel-Cantelli.

\sloppy
\printbibliography
\vspace{1cm}

\noindent\textsf{(Michael Fleermann)\newline
FernUniversit\"at in Hagen\newline
Fakult\"at f\"ur Mathematik und Informatik\newline 
Universit\"atsstra\ss e 1\newline 
58084 Hagen}\newline
\textit{E-mail address:}
\texttt{michael.fleermann@fernuni-hagen.de}
\vspace{1cm}

\noindent\textsf{(Johannes Heiny)\newline
Ruhr-Universit\"at Bochum\newline
Fakult\"at f\"ur Mathematik\newline 
Universit\"atsstra\ss e 150\newline 
44801 Bochum}\newline
\textit{E-mail address:}
\texttt{johannes.heiny@rub.de}

\end{document}